\documentclass[graybox]{svmult}

\usepackage{mathptmx}       
\usepackage{helvet}         
\usepackage{courier}        
\usepackage{type1cm}        
\usepackage{makeidx}         
\usepackage{graphicx}        
\usepackage{multicol}        
\usepackage[bottom]{footmisc}

\makeindex             

\usepackage{amsmath}
\usepackage{amssymb}
\def\e{{\mathrm{e}}}
\def\S{{\mathcal S}}
\def\R{{\mathcal R}}

\begin{document}

\title*{Continuous-time autoregressive moving-average processes in Hilbert space}
\titlerunning{Hilbert-valued CARMA processes} 
\author{Fred Espen Benth and Andr\'e S\"uss}
\institute{Fred Espen Benth \at University of Oslo, Department of Mathematics, POBox 1053 Blindern, N-0316 Oslo, Norway \email{fredb@math.uio.no} \and
Andre S\"uss \at Universitat de Barcelona, Facultat de Matematiques, Gran Via, 585, 
E-08007 Barcelona, Spain \email{suess.andre@web.de}}
%
%
\maketitle

\abstract*{We introduce the class of continuous-time autoregressive moving-average (CARMA) processes in Hilbert spaces. As driving noises of these processes we consider L\'evy processes in Hilbert space. We provide the basic definitions, show relevant properties of these processes and establish the equivalents of CARMA processes on the real line. Finally, CARMA processes in Hilbert space are linked to the stochastic wave equation and functional autoregressive processes.}

\abstract{We introduce the class of continuous-time autoregressive moving-average (CARMA) processes in Hilbert spaces. As driving noises of these processes we consider L\'evy processes in Hilbert space. We provide the basic definitions, show relevant properties of these processes and establish the equivalents of CARMA processes on the real line. Finally, CARMA processes in Hilbert space are linked to the stochastic wave equation and functional autoregressive processes.}

\section{Introduction}
\label{sec:intro}
Continuous-time autoregressive moving-average processes, or CARMA for short, play an important role in modelling the stochastic dynamics of various phenomena like wind speed, temperature variations and 
economic indices. For example, based on such models, in \cite{Andresen} the author analyses fixed-income markets while in \cite{BSB} and \cite{HLC} the dynamics of weather factors at various locations in Europe and Asia are modelled. Finally, in \cite{kluppel}, \cite{bnbv} and \cite{Prok} continuous-time autoregressive models for commodity markets like power and oil are studied.  

CARMA processes constitute the continuous-time version of autoregressive moving-average time series models. In this paper we generalize these processes to a Hilbert space context. The crucial ingredient in the extension is a "multivariate" Ornstein-Uhlenbeck process with values in a Hilbert space. There already exists an analysis of infinite dimensional L\'evy-driven Ornstein-Uhlenbeck processes, and we
refer the reader to the survey \cite{applebaum}. Moreover, matrix-valued operators and their semigroups play an important role. In \cite{EN} a detailed semigroup theory for such operators is developed. We review some of the results from \cite{applebaum} and \cite{EN} in the context of Hilbert-valued CARMA processes, as well as providing some new results for these processes.  

Let us recall the definition of a real-valued CARMA process. We follow \cite{brockwell} and first introduce the multivariate Ornstein-Uhlenbeck process $\{\vec Z(t)\}_{t\geq 0}$ with values in $\mathbb R^p$ for $p\in\mathbb N$ by
\begin{equation}
\label{eq:CARMA1d}
\D \vec Z(t)=C_p\vec Z(t)\D t+\vec e_p\D L(t),\ \vec Z(0)=\vec Z_0\in\mathbb R^p.
\end{equation}
Here, $L$ is a one-dimensional square integrable L\'evy process with zero mean defined on a complete probability space
$(\Omega,\mathcal F, P)$ with filtration $\mathcal F=\{\mathcal F_t\}_{t\geq 0}$, satisfying the usual hypotheses. 
Furthermore, $\vec e_i$ is the $i$th canonical unit vector
in $\mathbb R^p$, $i=1,\ldots,p$. The $p\times p$ matrix $C_p$ takes the particular form
\begin{equation}
\label{eq:real-car-matrix}
C_p=\left[\begin{array}{ccccccc} 
0 & 1 &            0 & .  & . & . & 0 \\
0 & 0            & 1 & 0 & . & . & 0 \\
.  & .             & .            & .  & . & . & . \\
.  & .             & .            & .  & . & . & . \\
0 & .             & .            & . &  .&  . & 1 \\
-\alpha_p & -\alpha_{p-1} & . & . & . & . & -\alpha_1
\end{array}\right],
\end{equation}
for constants $\alpha_i>0$, $i=1,\ldots,p$.\footnote{The odd labelling of these constants is explained by the 
relationship with autoregressive time series, where $\alpha_i$ is related (but not one-to-one) to the regression of lag $i$, $i=1,\ldots,p$.} 

We define a continuous-time autoregressive process of order $p$ by 
\begin{equation}
X(t)=\vec e_1^{\top}\vec Z(t),\ t\geq 0,
\end{equation}
where $\vec x^{\top}$ means the transpose of $\vec x\in\mathbb R^p$. We say that $\{X(t)\}_{t\geq 0}$ is
a CAR($p$)-process.  For $q\in\mathbb N$ with $p>q$, we define a
CARMA($p,q$)-process by
\begin{equation}
\label{eq:carma-real}
X(t)=\vec b^{\top}\vec Z(t),\ t\geq 0,
\end{equation}
where $\vec b\in\mathbb R^p$ is the vector $\vec b=(b_0,b_1,...,b_{q-1},1,0,..,0)^{\top}\in\mathbb R^p$, where
$b_q=1$ and $b_i=0, i=q+1,\ldots,p-1$. Note that $\vec b=\vec e_1$ yields a
CAR($p$)-process. Using the Euler-Maryuama approximation scheme, the CARMA($p,q$)-process 
$\{X(t)\}_{t\geq 0}$ on a discretized time grid can be related to an autoregressive moving average time series process of order 
$p,q$ (see \cite[Eq.~(4.17)]{BSB}). An explicit dynamics of the CARMA($p,q$)-process $\{X(t)\}_{t\geq 0}$ is (see
e.g. \cite[Lemma~10.1]{BSBK})
\begin{equation}
X(t)=\vec b^{\top}\exp(tC_p)\vec Z_0+\int_0^t\vec b^{\top}\exp((t-s)C_p)\vec e_p\D L(s),
\end{equation}
where $\exp(tC_p)$ is the matrix exponential of $tC_p$, the matrix $C_p$ multiplied by time $t$. 
 
If $C_p$ has only eigenvalues with negative real part, then the process $X$ admits a limiting
distribution $\mu_X$ with characteristic exponent (see \cite{brockwell})
$$
\widehat{\mu}_X(z):=\lim_{t\rightarrow\infty}\log\mathbb E\left[\E^{\mathrm{i}zX(t)}\right]
=\int_0^{\infty}\psi_L\left(\vec{b}^{\top}\exp(sC_p)\vec{e}_p z\right)\D s.
$$
Here, $\Psi_L$ is the log-characteristic function of $L(1)$ (see e.g. \cite{applebaum-book})
and $\log$ the distinguished logarithm (see e.g. \cite[Lemma~7.6]{sato}).
In particular, if $L=\sigma B$ with $\sigma>0$ constant and $B$ a standard Brownian motion, we find
$$
\widehat{\mu}_X(z)=-\frac12z^2\sigma^2\int_0^{\infty}(\vec{b}^{\top}\exp(sC_p)\vec{e}_p)^2\D s\,,
$$  
and thus $X$ has a Gaussian limiting distribution $\mu_X$ with zero mean and variance $\sigma^2\int_0^{\infty}(\vec{b}^{\top}\exp(sC_p)\vec{e}_p)^2\D s$. 

When $X$ admits a limiting distribution, we have a stationary representation of the process 
$X$ such that $X(t)\sim\mu_X$ for all $t\in\mathbb R$, namely,
\begin{equation}
X(t)=\int_{-\infty}^t\vec{b}^{\top}\exp((t-s)C_p)\vec{e}_p\D L(s),
\end{equation}
where $L$ is now a two-sided L\'evy process. This links CARMA($p,q$)-processes to the more
general class of L\'evy semistationary (LSS) processes, defined in \cite{bnbv} as
\begin{equation}
X(t):=\int_{-\infty}^tg(t-s)\sigma(s)\D L(s),
\end{equation}
for $g:\mathbb R_+\rightarrow\mathbb R$ being a measurable function and $\sigma$ a predictable
process such that $s\mapsto g(t-s)\sigma(s)$ for $s\leq t$ is integrable with respect to $L$.
Indeed, LSS processes are again a special case of so-called \textit{ambit fields}, which are 
spatio-temporal stochastic processes originally developed in \cite{BNS} for modelling turbulence.
In fact, the infinite dimensional CARMA processes that we are going to define in this paper will form
a subclass of ambit fields, as we will see in Section~\ref{sec:ambit}. We note that CARMA processes
with values in $\mathbb R^n$ have been defined and analysed by \cite{MarqStelz},
\cite{SchlemmStelz} and recently in \cite{Kevei}.


\section{Definition of CARMA processes in Hilbert space}
\label{sec:carma}

Given $p\in\mathbb N$, and let $H_i$ for $i=1,\ldots, p$ be separable Hilbert spaces with 
inner products denoted  by $\langle\cdot,\cdot\rangle_i$ and associated norms $|\cdot|_i$. We define
the product space $H:=H_1\times\ldots\times H_p$, which is again
a separable Hilbert space equipped with the inner product 
$\langle \vec x,\vec y\rangle:=\sum_{i=1}^p\langle x_i,y_i\rangle_i$ and the induced norm denoted $|\cdot|=\sum_{i=1}^n |\cdot|_i$ for 
$\vec x=(x_1,\ldots,x_p),\vec y=(y_1,\ldots,y_p)\in H$. The projection operator 
$\mathcal P_i:H\rightarrow H_i$ is defined as $\mathcal P_i\vec x=x_i$ for $\vec x\in H, i=1,\ldots,p$. 
It is straightforward to see that its adjoint $\mathcal P^*_i:H_i\rightarrow H$ is given by
$\mathcal P_i^*x=(0,\ldots,0,x,0,\ldots,0)$ for $x\in H_i$, where the $x$ appears in the $i$th 
coordinate of the vector consisting of $p$ elements.  If $U$ and $V$ are two separable Hilbert spaces, we denote
$L(U,V)$ the Banach space of bounded linear operators from $U$ to $V$, equipped with the operator norm
$\|\cdot\|_{\text{op}}$. The Hilbert-Schmidt norm for operators in $L(U,V)$ is denoted $\|\cdot\|_{\text{HS}}$, and $L_2(U,V)$ denotes the space of Hilbert-Schmidt operators. If $U=V$, we simply write $L(U)$ for $L(U,U)$. 

Let $A_i:H_{p+1-i}\rightarrow H_{p}, i=1,\ldots,p$ be $p$ (unbounded) densely defined linear operators, and
$I_i:H_{p+2-i}\rightarrow H_{p+1-i}, i=2,\ldots,p$ be another $p-1$ (unbounded) densely defined linear operators. 
Define the linear operator $\mathcal C_p: H\rightarrow H$ represented as a $p\times p$ matrix of operators
\begin{equation}
\label{carma-matrix-op-C}
\mathcal C_p=\left[\begin{array}{ccccccc} 
0 & I_p &            0 & .  & . & . & 0 \\
0 & 0            & I_{p-1} & 0 & . & . & 0 \\
.  & .             & .            & .  & . & . & . \\
.  & .             & .            & .  & . & . & . \\
0 & .             & .            & . &  .&  . & I_2 \\
A_p & A_{p-1} & . & . & . & . & A_1
\end{array}\right].
\end{equation}
Since the $A_i$'s and $I_i$'s are densely defined, $\mathcal C_p$ has domain 
$$
Dom(\mathcal C_p)=Dom(A_p)\times(Dom(A_{p-1})\cap Dom(I_p))\times\ldots \times (Dom(A_1)\cap Dom(I_2)),
$$
which we suppose is dense in $H$. 
We note in passing that typically, $H_1=H_2=\ldots=H_p$ and $I_i=\text{Id}$, the identity operator on $H_i$, $i=1,\ldots,p$. Then
$Dom(\mathcal C_p)=Dom(A_p)\times Dom(A_{p-1})\times\ldots \times Dom(A_1)$, which is dense in $H$.

On a complete probability space $(\Omega,\mathcal F,P)$ with filtration $\mathcal F=\{\mathcal F_t\}_{t\geq 0}$ satisfying the usual
hypotheses, denote by $L:=\{L(t)\}_{t\geq 0}$ a zero-mean square-integrable $H_p$-valued L\'evy process with covariance operator 
$Q$ (i.e., a symmetric non-negative definite trace class operator), see e.g. \cite[Sect.~4.9]{PZ}. 
Consider the following stochastic differential equation. For $t\geq 0$,
\begin{equation}
\label{def-OU-H}
\D\vec Z(t)=\mathcal C_p\vec Z(t)\D t+\mathcal P^*_p\D L(t),\ \vec Z(0):=\vec Z_0\in H.
\end{equation}
The next proposition states an explicit expression for $\vec Z:=\{\vec Z(t)\}_{t\geq 0}$:
\begin{proposition}
\label{prop:OU-mild-sol}
Assume that $\mathcal C_p$ defined in \eqref{carma-matrix-op-C} is the generator of a $C_0$-semigroup
$\{\mathcal S_p(t)\}_{t\geq 0}$ on $H$. Then the $H$-valued stochastic process 
$\vec Z$ given by
$$
\vec Z(t)=\mathcal S_p(t)\vec Z_0+\int_0^t\mathcal S_p(t-s)\mathcal P_p^*\D L(s)
$$
is the unique mild solution of \eqref{def-OU-H}. 
\end{proposition}
\begin{proof}
We have that $\mathcal S_p(t-s)\mathcal P_p^*\in L(H_p,H)$, and moreover, since $\|\mathcal P_p^*\|_{\text{op}}=1$ it follows
$$
\|\mathcal S_p(t-s)\mathcal P_p^*Q^{1/2}\|_{\text{HS}}\leq\|\mathcal S_p(t-s)\|_{\text{op}}
\|\mathcal P_{p}^*\|_{\text{op}}\|Q^{1/2}\|_{\text{HS}}\leq K \e^{c(t-s)}\|Q^{1/2}\|_{\text{HS}}
$$
by the general exponential growth bound on the operator norm of a $C_0$-semigroup (see e.g. \cite[Prop.~I.5.5]{EN}). Thus, for all $t\geq 0$,
$$
\int_0^t\|\mathcal S_p(t-s)\mathcal P_p^*Q^{1/2}\|^2_{\text{HS}}\D s\leq \frac{K}{2c}
\e^{2ct}\|Q^{1/2}\|^2_{\text{HS}}<\infty
$$
because $Q$ is trace class by assumption. The stochastic integral with respect to $L$ in the definition of $\vec Z$ is 
therefore well-defined. 
Hence, the result follows directly from the definition 
of mild solutions in \cite[Def.~9.5]{PZ}.
\end{proof}

From now on we restrict our attention to operators $\mathcal C_p$ in \eqref{carma-matrix-op-C} which admit a 
$C_0$-semigroup $\{\mathcal S_p(t)\}_{t\geq 0}$. We remark in passing that in the next section we will provide a recursive definition 
of the semigroup $\{\mathcal S_p(t)\}_{t\geq 0}$ in a special situation where all involved operators are bounded except
$A_1$. 

A CARMA process with values in a Hilbert space is defined next:
\begin{definition}
\label{def:carma}
Let $U$ be a separable Hilbert space. For $\mathcal L_U\in L(H,U)$, define the $U$-valued 
stochastic process $X:=\{X(t)\}_{t\geq 0}$ by
$$
X(t):=\mathcal L_U\vec Z(t), t\geq 0,
$$
for $\vec Z(t)$ defined in \eqref{def-OU-H}. We call $\{X(t)\}_{t\geq 0}$ a CARMA($p,U,\mathcal L_U$)-process.
\end{definition}
Note that we do not have any $q$-parameter present in the definition, as in the real-valued case 
(recall \eqref{eq:carma-real}). Instead we introduce a Hilbert space and a linear operator as the "second" parameters in the
CARMA($p,U,\mathcal L_U$)-process. Indeed, the vector $\vec b$ in the real-valued CARMA($p,q$)-process
defined in \eqref{eq:carma-real} can be viewed as a linear operator from $\mathbb R^p$ into $\mathbb R$
by the scalar product operation $\mathbb R^p\ni\vec z\mapsto \vec b'\vec z\in\mathbb R$, or, by
choosing $U=H_i=\mathbb R$, $\mathcal L_{\mathbb R}\vec z=\vec b'\vec z$. This also demonstrates that
any real-valued CARMA($p,q$)-process is a CARMA($p,\mathbb R,\vec b'\cdot$)-process according to Definition~\ref{def:carma}.

From Proposition~\ref{prop:OU-mild-sol} we find that the explicit representation of $\{X(t)\}_{t\geq 0}$
is
\begin{equation}
\label{carma-explicit}
X(t)=\mathcal L_U\mathcal S_p(t)\vec Z_0+\int_0^t\mathcal L_U\mathcal S_p(t-s)\mathcal P_p^*\D L(s),
\end{equation}
for $t\geq 0$. Note that by linearity of the stochastic integral we can move the operator 
$\mathcal L_U$ inside. Furthermore, the stochastic integral is well-defined since $\mathcal L_U\in L(H,U)$
and thus has a finite operator norm.

The continuous-time autoregressive (CAR) processes constitute a particularly interesting subclass
of the CARMA($p,U,\mathcal L_U$)-processes:
\begin{definition}
\label{def:car}
The CARMA($p,H_1,\mathcal P_1$)-process $\{X(t)\}_{t\geq 0}$ from Definition~\ref{def:carma} is called
an $H_1$-valued CAR($p$)-process.
\end{definition}
The explicit dynamics of an $H_1$-valued CAR($p$)-process becomes
\begin{equation}
\label{car-explicit}
X(t)=\mathcal P_1\mathcal S_p(t)\vec Z_0+\int_0^t\mathcal P_1\mathcal S_p(t-s)\mathcal P_p^*\D L(s),
\end{equation}
for $t\geq 0$. In this paper we will be particularly focused on $H_1$-valued CAR($p$)-processes.

Remark that the process $\vec L:=\mathcal P_p^*L$ defines an $H$-valued L\'evy process which has mean zero and is
square integrable. Its covariance operator is easily seen to be $\mathcal P_p^*Q\mathcal P_p$. 

It is immediate to see that an $H_1$-valued CAR($1$) process is an Ornstein-Uhlenbeck process defined on $H_1$, with
$$
\D X(t)=A_1 X(t)\D t+\D L(t),
$$
and thus
$$
X(t)=\mathcal S_1(t)Z_0+\int_0^t\mathcal S_1(t-s)\D L(s),
$$
being its mild solution. 

An $H_1$-valued CAR($p$) process for $p>1$ can be viewed as a higher-order (indeed, a $p$th order) stochastic differential equation, as we now discuss. To this end, suppose that $Ran (A_q)\subset Dom (I_2)$ and $Ran(I_{q})\subset Dom(I_{q+1})$, and  
assume that there exist $p-1$ linear (unbounded) operators $B_1, B_2,\ldots, B_{p-1}$ such that
\begin{equation}
\label{eq:commute}
I_p\cdots I_2 A_q=B_q I_p I_{p-1}\cdots I_{q+1}.
\end{equation}
for $q=1,\ldots,p-1$. We note that $I_p\cdots I_2: H_p\rightarrow H_1$ and hence
$I_p\cdots I_2A_q:H_{p+1-q}\rightarrow H_1$. Moreover, 
$I_p\cdots I_{q+1}:H_{p+1-q}\rightarrow H_1$, and therefore $B_q:H_1\rightarrow H_1$. 
We suppose that $Dom(B_q)$ is so that $Dom(B_qI_pI_{p-1}\cdots I_{q+1})=Dom(A_q)$,
and we note that $Dom(A_q)$ is the domain of the operator $I_p\cdots I_2 A_q$.
For completeness,
we define the operator $B_p:H_1\rightarrow H_1$ as
\begin{equation}
\label{eq:Bp}
B_p:=I_p\cdots I_2 A_p.
\end{equation}
We see that this definition is consistent with the inductive relations in \eqref{eq:commute}
(we suppose that $B_p$ is a linear (possibly unbounded) operator with domain $Dom(B_p)=Dom(A_p)$).  
With this at hand,
we introduce the operator-valued $p$th-order polynomial $Q_p(\lambda)$ for $\lambda\in\mathbb C$,
\begin{equation}
\label{def:Q-poly}
Q_p(\lambda)=\lambda^p-B_1\lambda^{p-1}-B_2\lambda^{p-2}-\cdots-B_{p-1}\lambda-B_p.
\end{equation}
By definition, $X(t)=\mathcal P_1\mathbf Z(t)$, which is the first coordinate in the vector 
$\vec Z(t)=(Z_1(t),\ldots,Z_p(t))^{\top}\in H$. From \eqref{def-OU-H} and  the definition of the operator matrix $\mathcal C_p$ in \eqref{carma-matrix-op-C}, we find that
$Z_1'(t)=I_pZ_2(t), Z_2'(t)=I_{p-1}Z_3(t),\ldots, Z_{p-1}'(t)=I_2Z_p(t)$ and finally
$$
Z_p'(t)=A_pZ_1(t)+\cdots A_1 Z_1(t)+\dot{L}(t).
$$
Here, $\dot{L}(t)$ is the formal time derivative of $L$. By iteration, we find that 
$Z_1^{(q)}(t)=I_pI_{p-1}\cdots I_{p-(q-1)}Z_{q+1}(t)$ for $q=1,\ldots,p-1$. Thus,
\begin{align*}
Z_1^{(p)}&=\frac{d}{dt}Z_1^{(p-1)}=I_p\cdots I_2 Z_p'(t) \\
&=I_p\cdots I_2A_p Z_1(t)+
I_p\cdots I_2A_{p-1}Z_2(t)+\cdots+I_p\cdots I_2 A_1Z_p(t)+I_p\cdots I_2\dot{L}(t) \\
&=B_p Z_1(t)+B_{p-1}Z_1'(t)+B_{p-2}Z_1^{(2)}(t)+\ldots+B_1Z_1^{(p-1)}(t)+
I_p\cdots I_2\dot{L}(t).
\end{align*}
In the last equality we made use of \eqref{eq:commute} and \eqref{eq:Bp}. We find that an $H_1$-valued
CAR($p$) process $X(t)=\mathcal P_1\vec Z(t)$ can be viewed as a solution to the 
$p$th-order stochastic differential equation formally expressed by
\begin{equation}
\label{eq:car-pth-order-sde}
Q_p\left(\frac{d}{dt}\right)X(t)=I_p\cdots I_2\dot{L}(t).
\end{equation}
Re-expressing Eq.~\ref{eq:car-pth-order-sde} we find the stochastic differential equation
\begin{equation}
\D X^{(p-1)}(t)=\left(\sum_{q=1}^p B_qX^{(p-q)}(t)\right)\D t+I_p\cdots I_2\D L(t).
\end{equation}   
If $H_1=...=H_p$ and $\mathcal C_p$ is a bounded operator, then 
$B_q=I_q\cdots I_2A_q$ in \eqref{eq:commute} whenever $I_q\cdots I_2A_q$ commutes
with $I_p\cdots I_{q+1}$. In this sense the condition \eqref{eq:commute} is a specific commutation
relationship on $A_q$ and the operators $I_2,\ldots,I_p$. In the particular case $I_i=\text{Id}$ for $i=2,\ldots,p$, then we trivially have $A_q=B_q$ for $q=1,\ldots,p$. 

We end this section with showing that the stochastic wave equation can be viewed as an example of a Hilbert-valued CAR(2)-process. To this end, let $H_2:=L^2(0,1)$, the space of square-integrable functions on the unit interval, and consider the stochastic partial differential
equation
\begin{equation}
\label{eq:wave}
\frac{\partial^2 Y(t,x)}{\partial t^2}=\frac{\partial^2 Y(t,x)}{\partial x^2}+\frac{\partial L(t,x)}{\partial t},
\end{equation}
with $t\geq 0$ and $x\in(0,1)$. We can rephrase this wave equation as
\begin{equation}
\D\left[\begin{array}{c} Y(t,x) \\ \frac{\partial Y(t,x)}{\partial t}\end{array}\right]=\left[\begin{array}{cc} 0 & \text{Id} \\ \Delta & 0\end{array}\right] \left[\begin{array}{c} Y(t,x) \\ \frac{\partial Y(t,x)}{\partial t}\end{array}\right]\D t+\left[\begin{array}{c} 0\\ \D L(t,x)\end{array}\right],
\end{equation}
with $\Delta=\partial^2/\partial x^2$ being the Laplace operator. The eigenvectors $e_n(x):=\sqrt{2}\sin(\pi n x)$, $n\in\mathbb N$, for $\Delta$ form an orthonormal basis of $L^2(0,1)$. Introduce the Hilbert space $H_1$ as the subspace of $L^2(0,1)$ for which
$|f|_{1}^2:=\pi^2\sum_{n=1}^{\infty}n^2\langle f,e_n\rangle_2^2<\infty$. Following Example B.13 in \cite{PZ}, 
$$
\mathcal C_2=\left[\begin{array}{cc} 0 & \text{Id} \\ \Delta & 0\end{array}\right] 
$$
generates a $C_0$-semigroup $\mathcal S_2(t)$ on $H:=H_1\times H_2$.
The Laplace operator $\Delta$ is a self-adjoint negative definite operator on $H_1$. The semigroup $\mathcal S_2(t)$ can be represented
as
\begin{equation}
\label{eq:wave-semigroup}
\mathcal S_2(t)=\left[\begin{array}{cc} \cos((-\Delta)^{1/2}t) & (-\Delta)^{-1/2}\sin((-\Delta)^{1/2}t) \\
-(-\Delta)^{1/2}\sin((-\Delta)^{1/2}t) & \cos((-\Delta)^{1/2}t)\end{array}\right].
\end{equation}
In the previous equality, we define for a real-valued function $g$ the linear operator $g(\Delta)$ using functional calculus, i.e., $g(\Delta)f=\sum_{n=1}^{\infty}g(-\pi^2 n^2)\langle f,e_n\rangle_2e_n$ whenever this sum converges. These considerations show that the wave equation is a specific example of a CAR(2)-process.

\section{Analysis of CARMA processes}
\label{sec:sv}
In this section we derive some fundamental properties of CARMA processes in Hilbert spaces. 

\subsection{Distributional properties}
We state the conditional characteristic functional of a CARMA($p,U,\mathcal L_U$)-process in the next proposition.
\begin{proposition}
\label{prop:char-exp-carma}
Assume $X$ is a CARMA($p,U,\mathcal L_U$)-process. Then, for $x\in U$, 
\begin{align*}
\mathbb E\left[\e^{\mathrm{i}\langle X(t),x\rangle_U}\,\vert\,\mathcal F_s\right]&=\exp\left(\mathrm{i}\langle\mathcal L_U\mathcal S_p(t)\vec Z_0,x\rangle_U+\int_0^{t-s}\psi_L\left(\mathcal P_p\mathcal S^*_p(u)\mathcal L_U^*x\right)\D u\right) \\
&\qquad\times\exp\left(\mathrm{i}\langle\int_0^s\mathcal L_U\mathcal S_p(t-u)\mathcal P_p^*\D L(u),x\rangle_U\right),
\end{align*}
for $0\leq s\leq t$. Here, $\psi_L$ is the characteristic exponent of the L\'evy process $L$. 
\end{proposition}
\begin{proof}
From \eqref{carma-explicit} it holds for $0\leq s\leq t$,
\begin{align*}
X(t)&=\mathcal L_U\mathcal S_p(t)\vec Z_0+\int_0^s\mathcal L_U\mathcal S_p(t-u)\mathcal P_p^*\D L(u)+\int_s^t\mathcal L_U\mathcal S_p(t-u)\mathcal P_p^*\D L(u).
\end{align*}
The L\'evy process has independent increments, and $\mathcal F_s$-measurability of the first stochastic integral thus yields
\begin{align*}
\mathbb E\left[\e^{\mathrm{i}\langle X(t),x\rangle_U}\,\vert\,\mathcal F_s\right]&=\exp\left(\mathrm{i}\langle\mathcal L_U\mathcal S_p(t)\vec Z_0,x\rangle_U+\mathrm{i}\langle\int_0^s\mathcal L_U\mathcal S_p(t-u)\mathcal P_p^*\D L(u),x\rangle_U\right) \\
&\qquad\times\mathbb E\left[\exp\left(\mathrm{i}\langle\int_s^t\mathcal L_U\mathcal S_p(t-u)\mathcal P_p^*\D L(u),x\rangle_U\right)\right].
\end{align*}
The result follows from \cite[Chapter 4]{PZ}.
\end{proof}
Suppose now that $L=W$, an $H_p$-valued Wiener process. Then the characteristic exponent is
$$
\psi_W(h)=-\frac12\langle Qh,h\rangle_p,
$$
for $h\in H_p$. Hence, from Prop.~\ref{prop:char-exp-carma} it follows that,
\begin{align*}
\mathbb E\left[\e^{\mathrm{i}\langle X(t),x\rangle_U}\,\vert\,\mathcal F_s\right]&=\exp\left(\mathrm{i}\langle\mathcal L_U\mathcal S_p(t)\vec Z_0,x\rangle_U+\mathrm{i}\langle\int_0^s\mathcal L_U\mathcal S_p(t-u)\mathcal P_p^*\D W(u),x\rangle_U\right) \\
&\qquad\times\exp\left(-\frac12\int_0^{t-s}\langle\mathcal L_U\mathcal S_p(u)\mathcal P_p^*Q\mathcal P_p\mathcal S^*_p(u)\mathcal L_U^*x,x\rangle_U\D u\right)
\end{align*}
We find that $X(t)|\mathcal F_s$ for $s\leq t$ is a Gaussian process in $H_1$, with mean
$$
\mathbb E\left[X(t)\,|\,\mathcal F_s\right]=\mathcal L_U\mathcal S_p(t)\vec Z_0+\int_0^s\mathcal W_U\mathcal S_p(t-u)\mathcal P_p^*\D L(u)
$$ 
and covariance operator
$$
\text{Var}(X(t)|\mathcal F_s)=\int_0^{t-s}\mathcal L_U\mathcal S_p(u)\mathcal P_p^*Q\mathcal P_p\mathcal S^*_p(u)\mathcal L_U^*\D u,
$$
where the integral is interpreted in the Bochner sense. If the semigroup $\mathcal S_p(u)$
is exponentially stable, then $X(t)|\mathcal F_s$ admits a Gaussian limiting distribution with mean zero
and covariance operator
$$
\lim_{t\rightarrow\infty}\text{Var}(X(t)|\mathcal F_s)=\int_0^{\infty}\mathcal L_U\mathcal S_p(u)\mathcal P_p^*Q\mathcal P_p\mathcal S^*_p(u)\mathcal L_U^*\D u.
$$
This is the invariant measure of $X$. We remark in passing that measures on $H$ are defined on its 
Borel $\sigma$-algebra. 

In \cite{applebaum} there is an analysis of invariant measures of L\'evy-driven Ornstein-Uhlenbeck processes. We discuss this here in
the context of the Ornstein-Uhlenbeck process $\{\vec Z(t)\}_{t\geq 0}$ defined in 
\eqref{def-OU-H}. Assume $\mu_{\vec Z}$ is the invariant measure of $\{\vec Z(t)\}_{t\geq 0}$,
and recall the definition of its characteristic exponent $\widehat{\mu}_{\vec Z}(\vec x)$,
\begin{equation}
\widehat{\mu}_{\vec Z}(\vec x)=\log\mathbb E\left[\E^{\mathrm{i}\langle\vec x,\vec Z(t)\rangle}\right].
\end{equation}
Here, $\vec x\in H$ and $\log$ is the distinguished logarithm (see e.g. \cite[Lemma~7.6]{sato}). 
If $\vec Z_0\sim\mu_{\vec Z}$, then, in distribution, $\vec Z_0=\vec Z(t)$ for all $t\geq0$ and it follows that 
the characteristic exponent of $\mu_{\vec Z}$ satisfies,
\begin{equation}
\widehat{\mu}_{\vec Z}(\vec x)=\widehat{\mu}_{\vec Z}(\mathcal S_p^*(t)\vec x)+
\int_0^t\psi_L(\mathcal P_p\mathcal S^*_p(u)\vec x)\D u
\end{equation}
for any $\vec x\in H$ and $t\geq 0$. Following \cite{applebaum}, $\mu_{\vec Z}$ becomes an operator self-decomposable distribution since, 
\begin{equation}
\mu_{\vec Z}=\mathcal S_p(t)\mu_{\vec Z}\star\mu_t.
\end{equation}
Here, $\mu_t$ is the distribution of $\int_0^t\mathcal S_p(u)\mathcal P^*_p\D L(u)$,
$\star$ is the convolution product of measures and $\mathcal S_p(t)\mu_{\vec Z}:=\mu_{\vec Z}\circ\mathcal S_p(t)^{-1}$ is a probability distribution on $H$, 
given by
$$
\int_Hf(\vec x)(\mathcal S_p(t)\mu_{\vec Z})(\D\vec x)=\int_Hf(\mathcal S_p(t)^*\vec x)
\mu_{\vec Z}(\D\vec x),
$$ 
for any bounded measurable function $f:H\rightarrow\mathbb R$. 
If $\vec Z(t)\sim\mu_{\vec Z}$, then since
$$
\log\mathbb E\left[\E^{\mathrm{i}\langle\mathcal L_U\vec Z(t),x\rangle_U}\right]
=\log\mathbb E\left[\E^{\langle\vec Z(t),\mathcal L^*_Ux\rangle}\right]=
\widehat{\mu}_{\vec Z}(\mathcal L_U^*x),
$$
it follows that $\{X(t)\}_{t\geq 0}$ is stationary with distribution $\mu_X$ having characteristic
exponent $\widehat{\mu}_X(x)=\widehat{\mu}_{\vec Z}(\mathcal L_U^*x)$ for $x\in U$.  

We notice that $\mathcal C_p$ is a bounded operator on $H$ if and only if $A_i, i=1,\ldots,p$ and
$I_j, j=2,\ldots,p$ are bounded operators. In the case of $\mathcal C_p$ being bounded, we know from Thm.~I.3.14 in \cite{EN} that the semigroup $\{\mathcal S_p(t)\}_{t\geq 0}$
is exponentially stable if and only if $\text{Re}(\lambda)<0$ for all $\lambda\in\sigma(\mathcal C_p)$,
where $\sigma(\mathcal C_p)$ denotes the spectrum of the bounded operator $\mathcal C_p$. 
Recall from Section~\ref{sec:intro} that a real-valued CARMA process admits a limiting distribution if and only if all the eigenvalues of 
$C_p$ in Eq.~\ref{eq:real-car-matrix} have negative real part. 
In general, by Thm.~V.1.11 in \cite{EN}, the semigroup $\{\mathcal S_p(t)\}_{t\geq 0}$ is exponentially stable if and only if $\{\lambda\in\mathbb C\,|\,\text{Re}(\lambda)>0\}$ is a subset of the resolvent set $\rho(\mathcal C_p)$ of $\mathcal C_p$ and $\sup_{\text{Re}(\lambda)>0}\|R(\lambda,\mathcal C_p)\|<\infty$. Here, $R(\lambda,\mathcal C_p)$ is the resolvent of $\mathcal C_p$ for $\lambda\in\rho(\mathcal C_p)$.


\subsection{Semimartingale representation}

Let us study a semimartingale representation of the CAR($p$) process. We have the following proposition:
\begin{proposition}
For $p\in\mathbb N$ with $p>1$, assume that $\mathcal C_p$ defined in \eqref{carma-matrix-op-C} is the 
generator of a $C_0$-semigroup $\{\mathcal S_p(t)\}_{t\geq 0}$. Then the $H_1$-valued CAR($p$) process $X$ given in
Definition~\ref{def:car} has the representation
$$
X(t)=\mathcal P_1\mathcal S_p(t)\vec Z_0+\mathcal P_1\mathcal C_p\int_0^t\int_0^{u}\mathcal S_p(u-s)\mathcal P_p^*\D L(s)\D u,
$$
for all $t\geq 0$. 
\end{proposition}
\begin{proof}
From \cite[Ch. II, Lemma~1.3]{EN}, we have that 
$$
\mathcal S_p(t)=\text{Id}+\mathcal C_p\int_0^t\mathcal S_p(s)\D s.
$$
But for any $\vec x\in H$, it is simple to see that $\mathcal P_1\text{Id}\mathcal P_p^*\vec x=0$ when $p>1$. Therefore
it holds
$$
\mathcal P_1\mathcal S_p(t)\mathcal P^*_p=\mathcal P_1\mathcal C_p\int_0^t\mathcal S_p(s)\mathcal P_p^*\D s.
$$
The integral on the right-hand side is in Bochner sense as an integral of operators. After appealing to the stochastic Fubini theorem,
it follows from the explicit expression of $X(t)$ in \eqref{car-explicit}
\begin{align*}
X(t)&=\mathcal P_1\mathcal S_p(t)\vec Z_0+\int_0^t\mathcal P_1\mathcal S_p(t-s)\mathcal P_p^*\D L(s) \\
&=\mathcal P_1\mathcal S_p(t)\vec Z_0+\int_0^t\mathcal P_1\mathcal C_p\int_0^{t-s}\mathcal S_p(u)\mathcal P_p^*\D u\D L(s) \\
&=\mathcal P_1\mathcal S_p(t)\vec Z_0+\mathcal P_1\int_0^t\mathcal C_p\int_s^{t}\mathcal S_p(u-s)\mathcal P_p^*\D u\D L(s).
\end{align*} 
We know from \cite[Ch. II, Lemma~1.3]{EN} that $\int_s^t\mathcal S_p(u-s)\mathcal P_p^*\D u\in Dom(\mathcal C_p)$. We 
demonstrate that $\int_0^t\int_s^t\mathcal S_p(u-s)\mathcal P_p^*\D u\D L(s)\in Dom(\mathcal C_p)$: First we recall that
$\vec L=\mathcal P_p^*L$ is an $H$-valued square-integrable L\'evy process
with mean zero. From the semigroup property,
\begin{align*}
\frac1h&\left(\mathcal S_p(h)\int_0^t\int_s^t\mathcal S_p(u-s)\D u\D\vec L(s)-\int_0^t\int_s^t\mathcal S_p(u-s)\D u\D\vec L(s)\right) \\
&\qquad=\frac1h\int_0^t\int_s^t\mathcal S_p(u+h-s)\D u\D\vec L(s)-\frac1h\int_0^t\int_s^t\mathcal S_p(u-s)\D u\D\vec L(s) \\
&\qquad=\int_0^t\frac1h\int_{s+h}^{t+h}\mathcal S_p(v-s)\D v-\frac1h\int_s^t\mathcal S_p(v-s)\D s\D\vec L(s) \\
&\qquad=\int_0^t\frac1h\int_{t}^{t+h}\mathcal S_p(v-s)\D v-\frac1h\int_s^{s+h}\mathcal S_p(v-s)\D s\D\vec L(s) \\
&\qquad=\int_0^t\frac1h\int_0^h\mathcal S_p(u)\D u\mathcal S_p(t-s)-\frac1h\int_0^h\mathcal S_p(u)\D u\D\vec L(s) \\
&\qquad=\frac1h\int_0^h\mathcal S_p(u)\D u\left(\int_0^t\mathcal S_p(t-s)\D\vec L(s)-\vec L(t)\right).
\end{align*}
By the fundamental theorem of calculus for Bochner integrals, $(1/h)\int_0^h\mathcal S_p(u)\D u\rightarrow\text{Id}$ when
$h\downarrow 0$. Therefore, the limit exists and the claim follows. From this we find that 
\begin{align*}
X(t)&=\mathcal P_1\mathcal S_p(t)\vec Z_0+\mathcal P_1\mathcal C_p\int_0^t\int_s^{t}\mathcal S_p(u-s)\mathcal P_p^*\D u\D L(s) \\
&=\mathcal P_1\mathcal S_p(t)\vec Z_0+\mathcal P_1\mathcal C_p\int_0^t\int_0^{u}\mathcal S_p(u-s)\mathcal P_p^*\D L(s)\D u. 
\end{align*} 
In the last equality, we applied the stochastic Fubini Theorem (see e.g. \cite[Thm.~8.14]{PZ}).
Hence,
the result follows. 
\end{proof}
Note that if $\vec Z_0\in Dom(\mathcal C_p)$, then by \cite[Ch. II, Lemma~1.3]{EN} 
$t\mapsto \mathcal P_1\mathcal S_p(t)\vec Z_0$ are differentiable. Assuming that $\int_0^t\mathcal S_p(t-s)\mathcal P_p^*\D L(s)\in Dom(\mathcal C_p)$, it follows from the Proposition above that the paths $t\mapsto X(t), t\geq 0$ of
$X$ is differentiable, with
\begin{equation}
\label{eq:diff_path}
X'(t)=\mathcal P_1\mathcal C_p\mathcal S_p(t)\vec Z_0+\mathcal P_1\mathcal C_p\int_0^t\mathcal S_p(t-s)\mathcal P_p^*\D L(s),
\end{equation} 
for $t\geq 0$. The stochastic integral in \eqref{eq:diff_path} has RCLL (cadlag) paths, and therefore the $H_1$-valued CAR($p$)-processes for $p>1$ have differentiable paths being RCLL. If $L=W$, an $H_p$-valued Wiener process, then the stochastic integral has continuous paths and the paths of $X$ become continuously differentiable.  
We point out that $p>1$ is very different from $p=1$ in this respect, as the Ornstein-Uhlenbeck process
\begin{align*}
X(t)&=\mathcal S_1(t)Z_0+\int_0^t\mathcal S_1(t-s)\D L(s) \\
&=\mathcal S_1(t)Z_0+L(t)+\int_0^t\int_0^u\mathcal S_1(u-s)\D L(s)\D u,
\end{align*}
does not have differentiable paths except in the trivial case when the L\'evy process is simply a drift.
It is straightforward to define an $H_p$-valued  L\'evy process $L$ for which  
$\int_0^t\mathcal S_p(t-s)\mathcal P_p^*\D L(s)\in Dom(\mathcal C_p)$. For example, let $\widetilde{L}$ be an $\mathbb R$-valued
square-integrable L\'evy process with zero mean, and define $L=\widetilde{L}g$ for $g\in Dom(A_1)\cap Dom(I_2)$. Then
$\mathcal P_p^*L=(\mathcal P_p^*g)\widetilde{L}\in Dom(\mathcal C_p)$, and therefore 
$\int_0^t\mathcal S_p(t-s)(\mathcal P_p^*g)\D\widetilde L(s)\in Dom(\mathcal C_p)$ from \cite[Ch. II, Lemma~1.3]{EN}.
If we consider the particular case of the wave equation, as presented at the end of Sect.~\ref{sec:intro}, we have $A_1=0$ and
$I_2=\text{Id}$, and thus we can choose any $g\in H_2$. In this case we can conclude that the paths of the solution of the wave 
equation are RCLL.

\subsection{Semigroup representation}

We study a recursive representation of the $C_0$-semigroup $\{\mathcal S_p(t)\}_{t\geq 0}$ with $\mathcal C_p$ as generator, where 
we recall $\mathcal C_p$ from \eqref{carma-matrix-op-C}. The following result is known as the variation-of-constants formula
(see e.g. \cite[Appendix B.1.1 and Thm. B.5]{PZ}) and turns out to be convenient when expressing the semigroup for
$\mathcal C_p$. 
\begin{proposition}
\label{prop:basic-semigroup-result}
Let $\mathcal A$ be a linear operator on $H$ being the generator of a $C_0$-semigroup $\{\mathcal S_{\mathcal A}(t)\}_{t\geq 0}$. 
Assume that $\mathcal B\in L(H)$.
Then the
operator $\mathcal A+\mathcal B: Dom(\mathcal A)\rightarrow H$ is the generator of the $C_0$-semigroup 
$\{\mathcal S(t)\}_{t\geq 0}$ defined by
$$
\mathcal S(t)=\mathcal S_{\mathcal A}(t)+\mathcal R(t),
$$
where
$$
\mathcal R(t)=\sum_{n=1}^{\infty}\mathcal R_n(t),
$$
and 
$$
\mathcal R_{n+1}(t)=\int_0^t\mathcal S_{\mathcal A}(t-s)\mathcal B\mathcal R_n(s)\D s,
$$
for $n=0,1,2,\ldots,$ with $\mathcal R_0(t)=\mathcal S_{\mathcal A}(t)$. 
\end{proposition}
We apply the proposition above to give a recursive description of the $C_0$-semigroup of $\mathcal C_p$. 
\begin{proposition}
\label{prop:carma-generator}
Given the operator $\mathcal C_p$ defined in \eqref{carma-matrix-op-C} for $p\in\mathbb N$, where $\mathcal C_1=A_1$ is a densely defined
linear operator on $H_p$ (possibly unbounded) with $C_0$-semigroup $\{\mathcal S_1(t)\}_{t\geq 0}$. 
For $p>1$, assume that $I_p\in L(H_2,H_1)$, $A_p\in L(H_1,H_p)$ and $\mathcal C_{p-1}$ is a densely defined operator on $H_2\times...\times H_p$ with a $C_0$-semigroup $\{\mathcal S_{p-1}(t)\}_{t\geq 0}$, then 
$$
\mathcal S_p(t)=\mathcal S_{p-1}^+(t)+\sum_{n=1}^{\infty}\mathcal R_{n,p}(t),
$$
where $\mathcal R_{0,p}(t)=\mathcal S_{p-1}^+(t)$ and for $n=1,2,\ldots,$
$$
\mathcal R_{n+1,p}(t)=\int_0^t\mathcal S_{p-1}^+(t-s)\mathcal B_p\mathcal R_{n,p}(s)\D s.
$$
Here, $\mathcal B_p\in L(H)$ is
$$
\mathcal B_p=\left[\begin{array}{ccccc} 
0 & I_p & 0 & . . . & 0 \\ 
0 & 0    & 0 & . . . & 0 \\
.  & .    & .  & . . . & . \\
 .  & .    & .  & . . . & . \\
A_p & 0 & . & . . . & 0
\end{array}\right]
$$
and $\mathcal S^+_{p-1}\in L(H)$
$$
\mathcal S^+_{p-1}=\left[\begin{array}{cccc} 
\text{Id} & 0 & . . . & 0 \\ 
0 &   &   &  \\
.  &   &   &  \\
.  &   & \mathcal S_{p-1}(t) & \\
0  &   &   &  
\end{array}\right]
$$
for $\text{Id}$ being the identity operator on $H_1$. 
\end{proposition}
\begin{proof}
By assumption, $I_p\in L(H_2,H_1)$ and $A_p\in L(H_1,H_p)$, and thus $\mathcal B_p\in L(H)$. Define 
$$
\mathcal A_p=\left[\begin{array}{cccc} 
0 & 0 & . . . & 0 \\ 
0 &   &   &  \\
.  &   &   &  \\
.  &   & \mathcal C_{p-1} & \\
0  &   &   &  \end{array}\right].
$$
Then, $\mathcal A_p+\mathcal B_p=\mathcal C_p$. Moreover, $\{\mathcal S_{p-1}^+(t)\}_{t\geq 0}$ is the
$C_0$-semigroup of $\mathcal A_p$. Hence, the result follows from Proposition~\ref{prop:basic-semigroup-result}. 
\end{proof}

As an example, consider $p=3$. Then we have 
$$
\mathcal C_3=\left[\begin{array}{ccc} 
0 & I_3 & 0 \\
0 & 0 & I_2 \\
A_3 & A_2 & A_1 
\end{array}\right].
$$
First, $\mathcal C_1=A_1$ is a (possibly unbounded) operator on $H_3$, with $C_0$-semigroup
$\{\mathcal S_1(t)\}_{t\geq 0}\subset L(H_3)$. Next, let
$$
\mathcal B_2=\left[\begin{array}{cc} 
0 & I_2 \\
A_2 & 0
\end{array}\right]
$$ 
where we assume $I_2\in L(H_3,H_2)$ and $A_2\in L(H_2,H_3)$ to have $\mathcal B_2\in L(H_2\times H_3)$. With
$$
\mathcal S_1^+(t)=\left[\begin{array}{cc}
\text{Id} & 0 \\
0 & \mathcal S_1(t) 
\end{array}\right],
$$
which defines a $C_0$-semigroup on $L(H_2\times H_3)$ with generator 
$$
\mathcal A_2=\left[\begin{array}{cc} 
0 & 0 \\
0 & A_1
\end{array}\right],
$$
we obtain 
$$
\mathcal S_2(t)=\mathcal S^+_1(t)+\sum_{n=1}^{\infty}\R_{n,2}(t)
$$
for $\R_{0,2}=\S_1^+(t)$ and
$$
\R_{n+1,2}(t)=\int_0^t\S_1^+(t_s)\mathcal B_2\R_{n,2}(s)\D s, n=1,2,\ldots.
$$
We note that $\{\mathcal S_2\}_{t\geq 0}\subset L(H_2\times H_3)$ is the $C_0$-semigroup with generator
$\mathcal C_2$ densely defined on $H_2\times H_3$. Finally, let
$$
\mathcal B_3=\left[\begin{array}{ccc}
0 & I_3 & 0 \\
0 & 0 & 0 \\
A_3 & 0 & 0
\end{array}\right]
$$
which is a bounded operator on $H$ after assuming $I_3\in L(H_2,H_1)$ and $A_3\in L(H_1,H_3)$. With
$$
\mathcal S_2^+(t)=\left[\begin{array}{cccc}
\text{Id} & 0 & & 0 \\
0 & & &  \\
  & & \mathcal S_2(t) &  \\
0 &  & & 
\end{array}\right],
$$
which is a $C_0$-semigroup on $L(H)$ with generator
$$
\mathcal A_3=\left[\begin{array}{cccc}
0 & 0 & & 0 \\
0 & & &  \\
  & & \mathcal C_2 &  \\
0 &  & & 
\end{array}\right],
$$
we conclude with 
$$
\mathcal S_3(t)=\mathcal S_2^+(t)+\sum_{n=1}^{\infty}\R_{n,3}(t),
$$
where $\R_{0,3}(t)=\mathcal S_2^+(t)$ and
$$
\R_{n+1,3}(t)=\int_0^t\mathcal S_2^+(t-s)\mathcal B_3\mathcal R_{n,3}(s)\D s.
$$
From this example we see that $A_2, A_3, I_2$ and $I_3$ must all be bounded operators, while only 
$A_1$ is allowed to be unbounded. By recursion in Proposition~\ref{prop:carma-generator}, we see that 
we must have $I_i\in L(H_{p+2-i},H_{p+1-i})$ and $A_i\in L(H_{p+1-i},H_p), i=2,3,\ldots,p$, and $A_1:Dom(A_1)\rightarrow H_p$ can be an 
unbounded operator with densely defined domain $Dom(A_1)\subset H_p$. 
 
We remark that Ch.~III in \cite{EN} presents a deep theory for perturbations of generators $\mathcal A$ by operators
$\mathcal B$. Matrix operators of the kind $\mathcal C_p$ for $p=2$ has been analysed in, for example \cite{Trunk}, 
where conditions for analyticity of the semigroup $\{\mathcal S_2(t)\}_{t\geq 0}$ is studied.

\section{Applications of CARMA processes}
\label{sec:ambit}

Recall Proposition~\ref{prop:OU-mild-sol}, and let $t_i:=i\cdot\delta$ for $i=0,1,\ldots$ and a given $\delta>0$. Define further
$\vec z_i:=\vec Z(t_i)$. By the semigroup property of $\{\mathcal S_p(t)\}_{t\geq 0}$ it holds,
\begin{align*}
\vec z_{i+1}&=\mathcal S_p(t_{i+1})\vec z_0+\int_0^{t_{i+1}}\mathcal S_p(t_{i+1}-s)\mathcal P^*_p\D L(s) \\
&=\mathcal S_p(\delta)\mathcal S_p(t_i)\vec z_i+\mathcal S_p(\delta)\int_0^{t_i}\mathcal S_p(t_i-s)\mathcal P_p^*\D L(s) \\
&\qquad+\int_{t_i}^{t_{i+1}}\mathcal S_p(t_{i+1}-s)\mathcal P_p^*\D L(s) \\
&=\mathcal S_p(\delta)\vec z_i+\vec{\epsilon}_i,
\end{align*}
with
$$
\vec{\epsilon}_i:=\int_{t_i}^{t_{i+1}}\mathcal S_p(t_{i+1}-s)\mathcal P_p^*\D L(s).
$$
The process above has the form of a discrete-time AR($1$) process.
Obviously, $\mathcal S_p(\delta)\in L(H)$ and by the independent increment property of the $H_p$-valued L\'evy process
$L$, $\{\vec\epsilon_i\}_{i=0}^{\infty}$ is a sequence of independent $H$-valued random variables. Furthermore, $\mathbb E[\vec\epsilon_i]=0$ due to the
zero-mean hypothesis of $L$. Finally, we can compute the covariance operators of $\vec\epsilon_i$ by appealing to the It\^o isometry
(cf. \cite[Cor.~8.17]{PZ})
\begin{align*}
\mathbb E[\langle\vec\epsilon_i,\vec x\rangle\langle\vec\epsilon_i,\vec y\rangle]
&=\int_{t_i}^{t_{i+1}}\langle Q\mathcal P_p\mathcal S_p^*(t_{i+1}-s)\mathcal P_1^*\vec x,\mathcal P_p\mathcal S_p^*(t_{i+1}-s)\mathcal P_1^*\vec y\rangle\D s \\
&=\int_0^{\delta}\langle\mathcal P_1\mathcal S_p(s)\mathcal P_p^*Q\mathcal P_p\mathcal S_p^*(s)\mathcal P_1^*\vec x,\vec y\rangle\D s,
\end{align*}
where $\vec x,\vec y\in H$. Thus, $\vec\epsilon_i$ has covariance operator $\mathcal Q_{\vec\epsilon}$ independent of $i$ given by
$$
\mathcal Q_{\vec\epsilon}=\int_0^{\delta}\mathcal P_1\mathcal S_p(s)\mathcal P_p^* Q\mathcal P_p\mathcal S_p^*(s)\mathcal P_1^*\D s.
$$
Therefore, $\{\vec\epsilon_i\}_{i=0}^{\infty}$ is an {\it iid} sequence of $H$-valued random variables. 
Hence, the $H$-valued time series 
$\{\vec z_i\}_{i=0}^{\infty}$ is a 
so-called {\it linear process} according to \cite{bosq}.

Let us now focus on the $H_1$-valued CAR($p$) dynamics in Def.~\ref{def:car}, and see how this
process can be related to a times series in $H_1$. To this end, recall the operator-valued polynomial
$Q_p(\lambda)$ introduced in \eqref{def:Q-poly} and the formal $p$th-order stochastic differential equation in \eqref{eq:car-pth-order-sde}. Let $\Delta_{\delta}$ be the forward differencing 
operator with time step $\delta>0$. Moreover, we assume $\Delta_{\delta}^n$ to be the $n$th order forward differencing, defined as
$$
\Delta_{\delta}^n f(t)=\sum_{k=0}^{n}\binom{n}{k}(-1)^kf(t+(n-k)\delta)
$$
for a function $f$ and $n\in\mathbb N$. Obviously, $\Delta_{\delta}^1=\Delta_{\delta}$. Introduce 
the discrete time grid $t_i:=i\delta$, $i=0,1,2,\ldots$, and observe that
$$
\frac{1}{\delta}\Delta_{\delta}I_p\cdots I_2 L(t_i)=\frac1{\delta}I_p\cdots I_2(L(t_{i+1})-L(t_i)).
$$
Assuming that the increments of $L$ belongs to the domain of $I_p\cdots I_2$, we find that 
\begin{equation}
\label{eq:carma-noise-discrete}
\epsilon_i:=\frac1{\delta}I_p\cdots I_2(L(t_{i+1})-L(t_i))
\end{equation}
for $i=0,1,2,\ldots$ define an {\it iid} sequence of $H_1$-valued random variables. We remark that this follows
from the stationarity hypothesis of a L\'evy process saying that  
the increments $L(t_{i+1})-L(t_i)$ are distributed as $L(\delta)$. The random
variables $\epsilon_i, i=0,1,\ldots,$ will be the numerical approximation of the formal expression 
$I_p\cdots I_2\dot{L}(t_i)$. Finally, we define (formally) a time series
$\{x_i\}_{i=0}^{\infty}$ in $H_1$ by
\begin{equation}
\label{def:euler-car}
Q_p\left(\frac{\Delta_{\delta}}{\delta}\right)x_i=\epsilon_i.
\end{equation}
In this definition, we use the notation $x_i=x(t_i)$ when applying the forward differencing operator $\Delta_{\delta}$. The 
polynomial $Q_p$ involves the linear operators $B_1,\ldots,B_p$ that may not be everywhere defined. We define the domain
$Dom(B)\subset H_1$ by
\begin{equation}
Dom(B):=Dom(B_1)\cap\cdots\cap Dom(B_p),
\end{equation}
which we assume to be non-empty. This will form the natural domain for the time series $\{x_i\}_{i=0}^{\infty}$. 
\begin{proposition}
\label{prop:euler}
Assume that for any $y_1,\ldots,y_p\in Dom(B)$, $B_1y_1+\cdots B_py_p\in Dom(B)$. If 
$\{\epsilon_i\}_{i=0}^{\infty}\subset Dom(B)$ with $\epsilon_i$ defined in \eqref{eq:carma-noise-discrete} and
$x_0,\ldots,x_{p-1}\in Dom(B)$, then $\{x_i\}_{i=0}^{\infty}$ is an AR($p$) process
in $H_1$ with dynamics
$$
x_{i+p}=\sum_{q=1}^p\widetilde{B}_qx_{i+(p-q)}+\delta^p\epsilon_i
$$
where 
$$
\widetilde{B}_q=(-1)^{q+1}\binom{p}{q}\text{Id}+\sum_{k=1}^q\delta^kB_k(-1)^{q-k}\binom{p-k}{q-k},
q=1,\ldots,p,
$$
and $\text{Id}$ is the identity operator on $H_1$.
\end{proposition}
\begin{proof}
First we observe that the assumption $B_1y_1+\cdots+B_py_p\in Dom(B)$ for any $y_1,\ldots,y_p\in Dom(B)$ is equivalent
with $\widetilde{B}_1y_1+\cdots+\widetilde{B}_py_p\in Dom(B)$ for any $y_1,\ldots,y_p\in Dom(B)$ since $\widetilde{B}_q$
is a linear combination of $B_1,\ldots,B_q$. Thus, by the assumptions, we see that $x_{i}\in Dom(B)$ for all $i=0,1,2...$ and
the recursion for the time series dynamics is well-defined.

We next show that the time series $\{x_i\}_{i=0}^{\infty}$ is indeed given by the recursion in the Proposition. 
From the definition of $Q_p$ and the forward differencing operators, we find after isolating $x_{i+p}$
on the left hand side and the remaining terms on the right hand side in the definition 
in Eq.~\eqref{def:euler-car} that 
\begin{align*}
x_{i+p}&=-\sum_{q}^{p}(-1)^{q}\binom{p}{q}x_{i+(p-q)}+\sum_{q=1}^{p-1}\delta^qB_q
\left(\sum_{k=0}^{p-q}(-1)^k\binom{p-q}{k}x_{i+(p-q-k)}\right) \\
&\qquad\qquad+\delta^pB_p x_i+\delta^p\epsilon_i.
\end{align*}
Identifying terms for $x_{i+(p-1)}, x_{i+(p-2)},\ldots,x_i$ yields the result.
\end{proof}
The time series $\{x_i\}_{i=0}^{\infty}$ defined in \eqref{def:euler-car} can be viewed as the 
numerical approximation of the $H_1$-valued CAR($p$) process $X(t)$. Notice that for small
$\delta$ we find that $\mathcal S_p(\delta)\approx\delta\mathcal C_p+\text{Id}$. Using this 
approximation in the explicit representation of $\mathbf Z(t)$ in Prop.~\ref{prop:OU-mild-sol} 
will yield the same conclusion as in our discussion above. 

We remark that if the operators $B_1,\ldots,B_p$ are bounded, then $Dom(B)=H_1$. In this case, the time series
$\{x_i\}_{i=0}^{\infty}$ will be everywhere defined on $H_1$, and we do not need to impose any additional "domain
preservation" hypothesis. 

Let us consider an example where $p=3$, and $H_1=H_2=H_3$. Suppose that $I_i=\text{Id}$
for $i=1,2,3$ and recall from the discussion in Section~\ref{sec:carma} that in this case
$B_q=A_q$ for $q=1,2,3$. Using Prop.~\ref{prop:euler} yields that
$$
x_{i+3}=(3\text{Id}+A_1)x_{i+2}+(A_2-2A_1-3\text{Id})x_{i+1}+(\text{Id}+A_1-A_2+A_3)x_i+\epsilon_i
$$
when $\delta=1$. Here, $\epsilon_i=L(t_{i+1})-L(t_i)$ and thus being distributed as $L(1)$. This
formula is the analogy of Ex.~10.2 in \cite{BSBK}. Indeed, Prop.~\ref{prop:euler} is
the generalization of  \cite[Eq.~(4.17)]{BSB} to Hilbert space. 

The $H_1$-valued AR($p$)-process in Prop.~\ref{prop:euler} is called a {\it functional 
autoregressive} process of order $p$ (or, in short-hand notation, FAR($p$)-process) by 
\cite{bosq}. For example, \cite{KR} apply such models in a functional data
analysis of Eurodollar futures, where they find statistical evidence for a FAR(2) dynamics. 
At this point, we would also like to mention that the stochastic wave equation considered in Sect.~\ref{sec:intro} will 
be an AR(2) process with values in $H_1$ (or a FAR(2)-process). Indeed, since in this case $A_1=0$, $I_2=\text{Id}$ and 
$A_2=\Delta$, the Laplacian, we find that $B_1=0$ and $B_2=\Delta$, and hence,
$$
x_{i+2}=2\text{Id}x_{i+1}-(\text{Id}-\delta^2\Delta)x_i+\delta^2\epsilon_i,
$$ 
for $i=0,1,2...$ Obviously, this recursion is obtained by approximating the wave equation by the discrete second derivative in time. 

Recalling from \eqref{eq:wave-semigroup} the semigroup $\{\mathcal S_2(t)\}_{t\geq 0}$ of the wave equation, we see 
from \eqref{car-explicit}
that it has the representation (with initial condition $\vec Z_0=\vec 0$)
$$
X(t)=\int_0^t(-\Delta)^{-1/2}\sin((-\Delta)^{1/2}(t-s))\D L(s).
$$
Following the analysis in \cite{BE}, $X$ will be a Hilbert-valued ambit field. Ambit fields have attracted a great deal of
attention as random fields in time and space suitable for modelling turbulence, say (see the seminal paper \cite{BNS} on
ambit fields and turbulence).  As $L$ is a $L^2(0,1)$-valued L\'evy process, one can represent it in terms of the basis 
$\{e_n\}_{n=1}^{\infty}$, where $e_n(x)=\sqrt{2}\sin(\pi n x)$, as
$$
L(t,x)=\sum_{n=1}^{\infty}\ell_n(t)e_n(x),
$$
with $\ell_n(t):=\langle L(t,\cdot),e_n\rangle_2, n=1,2,\ldots$ being real-valued square-integrable L\'evy processes
with zero mean (see \cite[Sect.~4.8]{PZ}). Thus, the stochastic wave equation has the representation
$$
X(t,x)=\sum_{n=1}^{\infty}\frac{\sqrt{2}}{\pi n}\int_0^t\sin(\pi n(t-s))\D\ell_n(s) \sin(\pi n x)
$$ 
which becomes an example of an ambit field. Hilbert-valued CARMA($p,U,\mathcal L_U$)-processes provide us with a rich class
of ambit fields, as real-valued CARMA processes are specific cases of L\'evy semistationary processes (see e.g. \cite{bnbv,BSB,BSBK}).

\begin{acknowledgement}
Financial support from the project FINEWSTOCH, funded by the Norwegian Research Council, is gratefully acknowledged.
\end{acknowledgement}

\end{document}